\documentclass[a4paper,11pt]{amsart}

\pdfoutput=1

\usepackage[text={400pt,660pt},centering]{geometry}

\usepackage{comment}
\usepackage{graphicx}
\usepackage{amsthm, amssymb, amsmath, amsfonts, mathrsfs}
\usepackage{mathtools}
\usepackage{scalerel} 
\usepackage{stackrel}
\usepackage[colorlinks=true, pdfstartview=FitV, linkcolor=blue, citecolor=blue, urlcolor=blue,pagebackref=false]{hyperref}
\usepackage{esint} 
\usepackage{MnSymbol} 
\usepackage{booktabs} 
\usepackage{microtype}
\usepackage{cleveref}
\usepackage{pgfplots}

\usepackage{graphicx}
\usepackage{caption}
\usepackage{subcaption}
 
\pgfplotsset{compat = newest}
\newtheorem{proposition}{Proposition}
\newtheorem{theorem}[proposition]{Theorem}
\newtheorem{lemma}[proposition]{Lemma}

\theoremstyle{remark}

\theoremstyle{definition}

\numberwithin{equation}{section}
\numberwithin{proposition}{section}
\numberwithin{table}{section}

\title[Uniqueness of weak solutions for Hamilton-Jacobi equations]{Uniqueness of semi-concave weak Solutions for Hamilton-Jacobi Equations}
\author[V. \ ISSA]{Victor Issa}

\address[Victor Issa]{\'Ecole Normale Sup\'erieure de Lyon, Lyon, France}

\begin{document}

\begin{abstract}
It is well known that when the nonlinearity is convex, the Hamilton-Jacobi PDE admits a unique semi-convex weak solution, which is the viscosity solution. In this paper, motivated by problems arising from spin glasses, we show that if the Hamilton-Jacobi PDE with strictly convex nonlinearity and regular enough initial condition admits a semi-concave weak solution, then this solution is the viscosity solution.
\end{abstract}

\maketitle


\section{Introduction}
It is well known that the Hamilton-Jacobi equation with convex nonlinearity admits a unique semi-convex weak solution \cite{Barles,Evans}, and that this unique solution is equal to the viscosity solution. More recently, it was shown that the Hamilton-Jacobi equation with arbitrary nonlinearity and smooth initial condition admits at most one convex weak solution \cite{SIFRT} and that it is equal to the viscosity solution. Let $\mathcal{C}^{1,1}$ denote the space of differentiable functions $\mathbb{R}^d \to \mathbb{R}$ with Lipschitz gradient. In the present document, we prove an analog of those results for semi-concave weak solutions.

\begin{theorem}  \label{t.uniqueness}

Let $T^* \in  [0,+\infty)$, $H : \mathbb{R}^d \to \mathbb{R}$ be a $\mathcal{C}^2$ strictly convex function and let $f : [0,T^*] \times \mathbb{R}^d \to \mathbb{R}$ be a Lipschitz function. Assume that $u_0 = f(0,\cdot)$ is $\mathcal{C}^{1,1}$, the function $f$ satisfies $\partial_ t f - H(\nabla_x f) = 0$ almost everywhere and there exists a positive constant $c > 0$ such that for all $t \in [0,T^*]$ the function $x \mapsto c|x|^2-f(t,x)$ is convex. Then $f$ coincides on $[0,T^*] \times \mathbb{R}^d$ with the unique viscosity solution of
\begin{equation} \label{e.HJ}
    \begin{cases}
\partial_ t u - H(\nabla_x u) = 0 \\
u(0,\cdot)= u_0.
\end{cases}
\end{equation}
\end{theorem}

If a function is both semi-concave and semi-convex on $\mathbb{R}^d$, then it is $\mathcal{C}^{1,1}$ \cite[Corollary 3.3.8]{scfunctions}.
In Theorem \ref{t.uniqueness}, since $f$ is equal to the viscosity solution, it is both semi-concave and semi-convex. As a consequence, $f$ is $\mathcal{C}^{1,1}$ and so $f$ is a classical solution of \eqref{e.HJ}. 

Let $\mathcal{P}_2(\mathbb{R}_+)$ be the set of Borel probability measures on $\mathbb{R}_+$ with finite second moment. In \cite{hopflaxwar}, it was pointed out that the limit free energy of the SK model is related to the viscosity solution of
\begin{equation} \label{e.HJW}
      \begin{cases}
          \partial_t p - \int \left( \partial_\mu p \right)^2 d\mu = 0 &\textrm{on } \mathbb{R}_+ \times \mathcal{P}_2(\mathbb{R}_+),\\ 
          p(0,\cdot) = \psi &\textrm{on } \mathcal{P}_2(\mathbb{R}_+),
      \end{cases}
\end{equation}
where $\psi$ is the cascade transform of the uniform measure on $\{-1,1\}$ \cite[Definition~3.2]{JC1}. More precisely, let $p$ be the viscosity solution of \eqref{e.HJW}, the Parisi formula \cite{panchenko,PhysRevLett.43.1754} at inverse temperature $\beta$ for the SK model may be written \cite[Theorem 1.1]{hopflaxwar}
\begin{equation} \label{e.parisi}
    \lim_{N \to \infty} F_N(\beta) = -p(\beta^2/2,\delta_0) +\frac{\beta^2}{2} + \log 2.
\end{equation}
According to this observation, it may be possible to prove the Parisi formula relying mainly on PDE tools. Devising such a proof was attempted in \cite{JC1}, but the method only yielded the following tight bound, 
\begin{equation} \label{e.bound}
    \lim_{N \to \infty} F_N(\beta) \leq -p(\beta^2/2,\delta_0) +\frac{\beta^2}{2} + \log 2.
\end{equation}
One of the limitations that prevented a complete proof of the Parisi formula in \cite{JC1} was the lack of a selection principle applicable in the context of spin glasses. We hope that the selection principle Theorem \ref{t.uniqueness} may generalize to the Hamilton-Jacobi equations that arise when studying the SK model, as well as more general spin glass models with convex covariance structures. The PDE point of view proposed in \cite{hopflaxwar} can also yield results in the context of non-convex spin glasses. For example, an analog of \eqref{e.bound} has been proven this way for the bipartite model \cite{JC2}. For this type of models, the free energy has yet to be rigorously identified. In Section~\ref{s.counterexample}, we give an example of a non-convex nonlinearity for which Theorem~\ref{t.uniqueness} fails. Some new ideas are needed, to further develop this method in the context of non-convex spin glasses.

\section*{Notation} 
For every $x \in \mathbb{R}^d$, we denote by $x^i$ the $i$-th coordinate of $x$. Given $R > 0$, we denote by $B_R$ the ball of center $0$ and radius $R$ in $\mathbb{R}^d$, with respect to the sup norm. Let $g : \mathbb{R}_+ \times \mathbb{R}^d \to \mathbb{R}$ be a differentiable function and $t \in \mathbb{R}_+$, we denote by $g^t$ the partial function $g(t,\cdot)$. The gradient of $g$ with respect to the space variable $x$ is denoted by $\nabla_x g$. In particular, $\nabla_x g  : \mathbb{R}_+ \times \mathbb{R}^d \to \mathbb{R}^d$. The notation $\nabla$ is reserved for the gradient with respect to all the variables. For example, here we have $\nabla g = (\partial_t g, \nabla_x g)$. We denote by $\mathcal{L}^d$ the Lebesgue measure on $\mathbb{R}^d$. If $f$ is a Lipschitz function, we denote by Lip$(f)$ the optimal Lipschitz constant of $f$. The push forward of a measure $\nu$ by a function $f$ is written $f_*\nu$. When a measure $\nu_1$ is absolutely continuous with respect to another measure, $\nu_2$ we write $\nu_1 \ll \nu_2$. For $\varepsilon >0$ we denote by $\eta_\varepsilon$ a molifier, for example $\eta_\varepsilon(x) = \frac{1}{\varepsilon^d} \eta(\frac{x}{\varepsilon})$ where 
\begin{equation*} \label{d.molifier}
    \eta(x) = \begin{cases}
        C \exp \left( \frac{1}{|x|^2-1} \right) &\mathrm{ if } \; |x| < 1, \\ 
        0 &\mathrm{ otherwise},
        \end{cases}
\end{equation*}
and $C$ is the constant such that $\int \eta(x) dx = 1$.
\section{Proof Sketch} \label{s.proofsketch}

In this section, we will outline the strategy we will use to prove Theorem \ref{t.uniqueness}. The following seemingly weaker result actually implies Theorem~\ref{t.uniqueness}.

\begin{theorem}  \label{t.wuniqueness}

Let $T^* \in [0,+\infty)$, $H : \mathbb{R}^d \to \mathbb{R}$ be a $\mathcal{C}^2$ strictly convex function and let $f : [0,T^*] \times \mathbb{R}^d \to \mathbb{R}$ be a Lipschitz function. Assume that $u_0 = f(0,\cdot)$ is $\mathcal{C}^{1,1}$, the function $f$ satisfies $\partial_ t f - H(\nabla_x f) = 0$ almost everywhere and there exists a positive constant $c > 0$ such that for every $t \in [0,T^*]$ the function $x \mapsto c|x|^2-f(t,x)$ is convex. Then, there exists $T \in (0,T^*]$ such that $f = u$ on $[0,T) \times \mathbb{R}^d$ where $u$ is the viscosity solution of \eqref{e.HJ}.

\end{theorem}

Note that the hypotheses on $H$ and $u_0$ in Theorems \ref{t.uniqueness} and \ref{t.wuniqueness} guarantee the existence and uniqueness of a viscosity solution \cite[Theorem 7.1]{Barles}. Furthermore, the viscosity solution is Lipschitz and semi-convex \cite[Theorems 8.2 \&~8.3]{Barles}. We will use those facts during the proofs of Theorems \ref{t.uniqueness} and \ref{t.wuniqueness}. Let us start by proving that Theorem \ref{t.uniqueness} is a consequence of Theorem \ref{t.wuniqueness}. 
 
\begin{proof}[Proof of Theorem \ref{t.uniqueness} using Theorem \ref{t.wuniqueness}]
Suppose Theorem \ref{t.wuniqueness} is true, let $u$ be the viscosity solution of \eqref{e.HJ} and define
\begin{equation*}
    T_0 = \sup \{T \in [0,T^*), \, f=u \textrm{ on } [0,T) \times \mathbb{R}^d \}.
\end{equation*}
Suppose $T_0 < T^*$, by definition of $T_0$ we have $f = u$ on $[0,T_0) \times \mathbb{R}^d$ and even on $[0,T_0] \times \mathbb{R}^d$ since both $u$ and $f$ are Lipschitz. Let $u_1= f(T_0,\cdot) = u(T_0,\cdot)$ ,since $f(T_0,\cdot)$ is semi-concave and $u(T_0,\cdot)$ is semi-convex, the function $u_1$ is $\mathcal{C}^{1,1}$. Hence, $(t,x) \mapsto f(T_0+t,x)$ satisfies the hypothesis of Theorem \ref{t.wuniqueness}, so there exists $T > 0$ such that $f(T_0+t,x) = u(T_0+t,x)$ on $[0,T) \times \mathbb{R}^d$. Thus, $f = u$ on $[0,T_0+T) \times \mathbb{R}^d$ which contradicts the definition of $T_0$. In conclusion, under Theorem \ref{t.wuniqueness} we must have $T_0 = T^*$, which means that $f = u$.
\end{proof}

Let us now sketch the proof of Theorem \ref{t.wuniqueness}. We will proceed in two steps, and use arguments involving characteristic curves. If $g$ is a smooth solution of \eqref{e.HJ}, then differentiating \eqref{e.HJ} with respect to the $i$-th space coordinate yields
\begin{equation}
    \partial_t \partial_{x_i} g - \nabla H(\nabla_x g) \cdot \nabla_x(\partial_{x_i} g) = 0. 
\end{equation}
Hence, $\partial_{x_i} g$ solves a transport equation, with a vector field that depends on $g$. In particular, we may learn new information on $g$ by studying the solutions of the associated characteristic ordinary differential equation $\dot{\phi} = -\nabla H(\nabla_x g(t,\phi))$.

The first step of the proof of Theorem \ref{t.wuniqueness} will be to show that there exists $T \in (0,T^*]$ such that $f \geq u$ on $[0,T) \times \mathbb{R}^d$. This is done by studying the behavior of $f$ and $u$ along the solutions of 
\begin{equation} \label{e.ODEuu}
    \begin{cases}
    \dot{\phi} = -\nabla H(\nabla_x u(t,\phi)), \\
    \phi(0) = x,
    \end{cases}
\end{equation}
and showing that if $\phi$ is a solution of \eqref{e.ODEuu} then $t \mapsto (f-u)(t,\phi(t))$ is non-decreasing. The viscosity solution is $\mathcal{C}^{1,1}$ on $[0,T_0)$ for some $T_0 > 0$, so the existence of solutions to \eqref{e.ODEuu} in short time is not an issue for this step of the proof.

For the second step of the argument, we will use the same idea. However, this time we will show that $f-u$ is non-increasing along the solutions of
\begin{equation} \label{e.ODE}
    \begin{cases}
\dot{\phi} = -\nabla H(\nabla_xf(t,\phi)), \\
\phi(0) = x.
\end{cases}
\end{equation}
Here the existence of solutions to \eqref{e.ODE} is not clear since the right-hand side is not fully Lipschitz. This kind of problem has already been investigated in \cite{bvequation,onesided,LionsSeeger}. We will give a fully self-contained existence argument, relying on ideas borrowed from \cite[Theorem 6.2]{bvequation}. The precise construction is given in Section \ref{s.constructionofx} and uses some results from the theory of Young measures (see Theorem \ref{t.youngmeasure}). To construct a solution of \eqref{e.ODE}, we will build solutions of problems that approximate \eqref{e.ODE} and have smooth right-hand sides. Then, using the semi-concavity of $f$, We will check that this sequence of solutions converges to a solution of \eqref{e.ODE} in an appropriate sense.

Once the existence of solutions is obtained, we can show that $f \leq u$ on the set $E$ of points of $[0,T) \times \mathbb{R}^d$ reached by the solutions of \eqref{e.ODE} built previously. Finally, since we already know that $f \geq u$ on $[0,T) \times \mathbb{R}^d$, we will deduce that $f = u$ on $E$, and using uniqueness of solutions in short time for \eqref{e.ODEuu} we will deduce that $E = [0,T) \times \mathbb{R}^d$ which will end the proof of Theorem \ref{t.wuniqueness}. This final part is treated in Section \ref{s.endproof}. Note, if $c$ is a semi-concavity constant of $f$ and $H(p) = |p|^2/2$, then  the right-hand side of \eqref{e.ODE} satisfies, 
\begin{equation} \label{e.ouronesided}
    \left( -\nabla_xf(t,y) - (-\nabla_xf(t,x)) \right) \cdot \left( y - x\right) \geq -c|y - x|^2.
\end{equation}
This means that the distance between two solutions of \eqref{e.ODE} started at $x$ and $y$ is lower bounded by $|x-y|\exp(\frac{-ct}{2})$. So, the solutions of \eqref{e.ODE} tend to repel each other.
\section{Lower bound} \label{s.firstineq}

There exists a short time $T_0 \in (0,T^*]$ such that $u$ is $\mathcal{C}^{1,1}$ on $[0,T_0) \times \mathbb{R}^d$ \cite[Proposition 12.1]{PLL}. In this section, we will use this smoothness result to prove the existence of $T \in (0,T^*]$ such that $f \geq u$ on $[0,T) \times \mathbb{R}^d$. Define $a = - \nabla H (\nabla_x u)$ so that the space derivatives of $u$ satisfy the following almost everywhere on $[0,T_0) \times \mathbb{R}^d$,
\begin{equation}
    \partial_t (\partial_{x_i} u) + a \cdot \nabla \partial_{x_i} u = 0.
\end{equation}
We are interested in the characteristic equation of this PDE, namely 
\begin{equation} \label{e.ODEu}
    \dot{\phi}(t) = a(t,\phi(t)).
\end{equation}
Let us now define $T = \min(T_0, 1/\mathrm{Lip} (\nabla H(\nabla u_0)))$, for every $t < T$ and $x \in \mathbb{R}^d$, let $W(t,x) = x - t \nabla H (\nabla u_0(x))$. For every $t < T$, the map $W^t$ is bi-Lipschitz and equal to $id_{\mathbb{R}^d}$ up to a Lipschitz perturbation with Lipschitz constant smaller than $1$. From this, we deduce the following lemma, which will enable us to prove that $f-u$ is a decreasing function of time along $W$.

\begin{lemma} \label{l.propw}
On $[0,T)$, the push forward of $\mathcal{L}^{d+1}$ by $(t,x) \mapsto (t,W(t,x))$ is absolutely continuous with respect to $\mathcal{L}^{d+1}$ and for every $t < T$, the function $W^t$ is surjective.
\end{lemma}

\begin{proof}

Define $W_0(t,x) = (t,W(t,x))$, let $B \subset [0,T) \times \mathbb{R}^d$ be a set of measure $0$. We have $(W_0)^{-1}(B) = \{ (t,x), (t,W(t,x)) \in B\}$, define $B^t = \{x, (t,x) \in B \} \subset \mathbb{R}^d$. By Fubini's theorem, we have 

\begin{align*}
    \mathcal{L}^{d+1}((W_0)^{-1}(B)) &= \int_0^T \int_{\mathbb{R}^d} 1_B(t,W(t,x))dxdt \\
    &= \int_0^T \mathcal{L}^d((W^t)^{-1}(B^t))dt. 
\end{align*}
On the other hand, we have $0 = \mathcal{L}^{d+1}(B) = \int_0^T \mathcal{L}^d(B^t)dt$ so $\mathcal{L}^d(B^t) = 0$ $t$-almost everywhere. In particular, if we show that the pushforward of $\mathcal{L}^d$ by $W^t$ is absolutely continuous with respect to $\mathcal{L}^d$ and that $W^t$ is surjective, then the lemma is proven.

Since $t < T $, the function $x \mapsto W^t(x) - x$ is a contraction. Hence, it suffices to show the following. For every Lipschitz function $g : \mathbb{R}^d \to \mathbb{R}^d$ with Lipschitz constant $L < 1$ with respect to the sup norm, the function $h : x \mapsto x +g(x)$ is surjective and the pushforward of $\mathcal{L}^d$ by $h$ is absolutely continuous with respect to $\mathcal{L}^d$.

Let $y \in \mathbb{R}^d$, the function $x \mapsto y - g(x)$ has a unique fixed point $x^*$, and by construction $h(x^*) = y$, so $h$ is surjective. Now, we wish to show that $h_* \mathcal{L}^d \ll \mathcal{L}^d$. Let $|\cdot|_\infty$ denote the $\ell^\infty$ norm on $\mathbb{R}^d$, denote by $B(x,r) = x + (-r,r)^d$ the ball of center $x$ and radius $r$ in $\mathbb{R}^d$ with respect to $|\cdot|_\infty$. First, note that for every $x,y \in \mathbb{R}^d$ we have $|h(x) - h(y)|_\infty\geq |x-y|_\infty - |g(x)-g(y)|_\infty \geq (1-L)|x-y|_\infty$ with $1-L > 0$. Since $|h(x)-h(y)|_\infty \geq (1-L)|x-y|_\infty$ we have $h^{-1}(B(h(x),r)) \subset B(x,\frac{1}{1-L})$. In particular, for every ball $B$ using the surjectivity of $h$ we have $\mathcal{L}^d(h^{-1}(B)) \leq \frac{1}{(1-L)^d} \mathcal{L}^d(B)$. Now let $A \subset \mathbb{R}^d$ be a set of measure $0$. By definition of the outer Lebesgue measure, for every $\delta > 0$ there exists a sequence of balls $(B_n)_n$ such that, $A \subset \bigcup B_n$ and $\sum \mathcal{L}^d(B_n) < \delta$. So we have

\begin{align*}
  \mathcal{L}^d(h^{-1}(A)) &\leq \sum_n \mathcal{L}^d(h^{-1}(B_n)) \\
  &\leq \sum_n \frac{1}{(1-L)^d} \mathcal{L}^d(B_n) \\
  &< \frac{\delta}{(1-L)^d}.
\end{align*}
This is true for all $\delta > 0$, so $\mathcal{L}^d(h^{-1}(A)) = 0$. Thus, $h_* \mathcal{L}^d \ll \mathcal{L}^d$.
\end{proof}

\begin{lemma} \label{l.unicity}
  For $\mathcal{L}^d$-almost every $x \in \mathbb{R}^d$, $W(\cdot,x)$ is the unique integral solution on $[0,T)$ of \eqref{e.ODEu} with initial condition $\phi(0) = x$.
\end{lemma}

\begin{proof} 
  Since $a$ is Lipschitz on $[0,T) \times \mathbb{R}^d$, the Picard–Lindelöf theorem applies and \eqref{e.ODEu} admits a unique maximal solution. Let
  \begin{equation}
    D = \{ (t,x) \in [0,T) \times \mathbb{R}^d, \textrm{$\nabla u$ is not differentiable at $(t,x)$} \},
  \end{equation}
  by Rademacher's theorem the set $D$ is $\mathcal{L}^{d+1}$ negligible. According to Lemma \ref{l.propw}, this means that the set $\{ (t,x) \in [0,T) \times \mathbb{R}^d, \textrm{$\nabla u$ is not differentiable at $(t,W^t(x))$} \}$ is also $\mathcal{L}^{d+1}$ negligible. In particular, for $\mathcal{L}^d$-almost every $x \in \mathbb{R}^d$, the function $t \mapsto \nabla u(t,W(t,x))$ is differentiable $\mathcal{L}^1$-almost everywhere on $[0,T)$. Let $t \in [0,T)$ such that $\nabla u$ is differentiable at $(t,W(t,x))$, we have 
  \begin{equation}
     \frac{d}{dt} \nabla_x u(t,W(t,x)) = \nabla_x u_t(t,W(t,x)) + a(t,W(t,x)) \cdot \nabla_x (\nabla_x u(t,W(t,x))) = 0.
  \end{equation}
  Since $t \mapsto \nabla_x u(t,W(t,x))$ is a Lipschitz function, it is equal to the integral of its derivative, so it is constant on $[0,T)$.
  It follows that, for every $t \in [0,T)$,
  \begin{equation}
    W(t,x) = x + \int_0^t \nabla H(\nabla u_0(x))d\tau = x + \int_0^t \nabla H(\nabla_x u(\tau,W(\tau,x))) d\tau.
  \end{equation}
  Thus, $W(\cdot,x)$ is the unique integral solution of \eqref{e.ODEu} on $[0,T)$.
\end{proof}

Let us now study the behavior of $f$ and $u$ along the solutions of \eqref{e.ODEu}. By Rademacher's theorem, $f$ and $u$ are differentiable almost everywhere. According to Lemma \ref{l.propw}, the push forward of $\mathcal{L}^{d+1}$ by $(t,x) \mapsto (t,W(t,x))$ is absolutely continuous with respect to $\mathcal{L}^{d+1}$. Hence, for $\mathcal{L}^d$-almost every $x \in \mathbb{R}^d$ the function $t \mapsto (f-u)(t,W(t,x))$ is differentiable $\mathcal{L}^1$-almost everywhere. Using the convexity of $H$ the following holds $\mathcal{L}^{d+1}$-almost everywhere,

\begin{align*}
  \frac{d}{dt} (f-u)(t,W(t,x)) &= \partial_t(f-u)(t,W(t,x)) - \nabla H( \nabla_x u(t,W(t,x))) \cdot (\nabla_x f - \nabla_x u)(t,W(t,x))\\
  &= \left[ H(\nabla_x f) - H(\nabla_x u) + \nabla H (\nabla_x u) \cdot( \nabla_x u - \nabla_x f) \right](t,W(t,x)) \\
  &\geq 0.
\end{align*}
Since $f$, $u$ and $W$ are Lipschitz functions, the mapping $t \mapsto (f-u)(t,W(t,x))$ is equal to the integral of its derivative. So, we have $(f-u)(t,W^t(x)) \geq 0$ for every $t < T$ and $\mathcal{L}^d$-almost every $x \in \mathbb{R}^d$. By continuity of $f$, $u$ and $W$ this last bound is actually true for every $t < T$ and for every $x \in \mathbb{R}^d$. According to Lemma \ref{l.propw}, for every $t < T$ the function $W^t$ is surjective. In conclusion, we have proven the following lemma.

\begin{lemma} \label{l.firstineq}
Let $f$ and $H$ be functions satisfying the hypothesis of Theorem \ref{t.wuniqueness}, and let $u$ be the viscosity solution of \eqref{e.HJ}. Let $T_0 > 0$ be such that $u \in \mathcal{C}^{1,1}([0,T_0) \times \mathbb{R}^d)$, and define $T = \min(T_0, 1/\mathrm{Lip} (\nabla H(\nabla u_0))$. Then, for every $t < T$ and $x \in \mathbb{R}^d$, $f(t,x) \geq u(t,x)$.
\end{lemma}
 
\section{Construction of the characteristics} \label{s.constructionofx}

As explained in Section \ref{s.proofsketch}, if we define $b = -\nabla H( \nabla_x f)$ and if $f$ is twice differentiable, we have for all $i$,
\begin{equation}
   \partial_t (\partial_{x_i} f) + b \cdot \nabla_x (\partial_{x_i} f) = 0.
\end{equation}
This means that, if they exist, the derivatives of $f$ each satisfy a transport equation with a vector field $b$ that depends on $f$. The aim of this section is to construct appropriate solutions of the associated characteristic equation,
\begin{equation} \label{e.chareqz}
  \dot{\phi}(t) = b(t,\phi(t)).
\end{equation}
Note that this equation can be studied even if $f$ is not twice differentiable, as there are only derivatives of order at most one that appear in the definition of the vector field $b$ and those always exist, at least almost everywhere, since $f$ is Lipschitz. We define $f_\varepsilon(t,x) = (f^t * \eta_\varepsilon)(x)$ and $b_\varepsilon = - \nabla H (\nabla_x f_\epsilon)$, so that $b_\varepsilon$ is a smooth approximation of $b$. Furthermore, if $\phi_\varepsilon$ is a maximal solution of
\begin{equation} \label{e.apprxoODE}
  \begin{cases}
\dot{\phi}_\varepsilon(t) = b_\varepsilon(t,\phi(t)), \\
\phi_\varepsilon(0) = x.
\end{cases}
\end{equation} 
Then, $|\phi_\varepsilon(t)| \leq x + t\|b_\varepsilon\|_{L^\infty([O,T^*] \times \mathbb{R}^d)}$, so the maximal solution of \eqref{e.apprxoODE} does not blow up in finite time, hence it is defined on $[0,T^*)$. Define $X_\varepsilon(t,x)$ as the value at $t \in [0,T^*)$ of the solution of \eqref{e.apprxoODE} started at $x \in \mathbb{R}^d$. We will show that the sequence $(X_\varepsilon)_\varepsilon$ converges, in some sense, to a probability measure on the set of solution of \eqref{e.chareqz} as $\varepsilon \to 0$. To do so, we will use the theory of Young measures. We start by showing that the $X^t_\varepsilon$'s do not expand too much the size of measurable sets. Recall that given $R \in (0,\infty]$, $B_R$ denotes the ball of center $0$ and radius $R$ in $\mathbb{R}^d$ with respect to the sup norm and with the understanding that $B_\infty = \mathbb{R}^d$.

\begin{lemma} \label{l.anticoncentration}

There exists a constant $c > 0$ depending on $H$ and $f$ such that, for every $R \in (0,\infty]$, $\varepsilon>0$, $t \in [0,T^*)$ and non-negative continuous function $\beta : \mathbb{R}^d \to \mathbb{R}$,
\begin{equation} \label{e.anticoncentration}
    \int_{B_R}\beta (X_\varepsilon^t(x))dx \leq e^{ct} \int_{B_{R + t|b|_\infty}} \beta(x) dx.
\end{equation}
Furthermore, for every measurable set $A \subset \mathbb{R}^d$, we have $\mathcal{L}^d((X_\varepsilon^t)^{-1}(A)) \leq e^{ct} \mathcal{L}^d(A)$ and $(X_\varepsilon)_*\mathcal{L}^d \ll \mathcal{L}^d$.
\end{lemma}

\begin{proof}

A change of variable yields
\begin{equation}
    \int_{B_R} \beta(X_\varepsilon^t(x))dx  = \int_{X^t_\varepsilon(B_R)} \beta(x) \det(\nabla_x X^t_\varepsilon((X_\varepsilon^t)^{-1}(x)))^{-1}dx.
\end{equation}
We have $X^t_\varepsilon(B_R) \subset B_{R + t|b|_\infty}$, so it is enough to show that $\det(\nabla_x X^t_\varepsilon(x))^{-1} \leq e^{ct}$. It follows from the definition of $X_\varepsilon$ that
\begin{equation}
    \frac{d}{dt} \nabla_x X_\varepsilon^t(x) = \nabla_x b_\varepsilon^t(X_\varepsilon^t(x)) \nabla_x X_\varepsilon^t(x).
\end{equation}
If $A(t)$ is a matrix that is a differentiable function of a parameter $t$, then $\mathrm{det}(A(t))$ is a differentiable function of $t$ and
\begin{equation}
    \frac{d}{dt} \mathrm{det}(A(t))~=~\mathrm{Tr}(\mathrm{adj}(A(t)) \frac{d}{dt} A(t)),
\end{equation}
where $\mathrm{adj}(A(t))$ is the adjugate of $A(t)$ (transpose of the cofactor matrix), in particular $A(t) \mathrm{adj}(A(t)) = \mathrm{det}(A(t)) I_d$. Applying this to the matrix $A(t) = \nabla_x X_\varepsilon^t(x)$ yields 

\begin{align*}
    \frac{d}{dt} \det (\nabla_x X_\varepsilon^t(x)) &= \mathrm{Tr}(\mathrm{adj}(\nabla_x X_\varepsilon^t(x)) \nabla_x b_\varepsilon^t(X_\varepsilon^t(x)) \nabla_x X_\varepsilon^t(x)) \\
    &= \mathrm{Tr}(\nabla_x X_\varepsilon^t(x)  \mathrm{adj}(\nabla_x X_\varepsilon^t(x)) \nabla_x b_\varepsilon^t(X_\varepsilon^t(x)) ) \\
    &= \det(\nabla_x X_\varepsilon^t(x))    \mathrm{div} \,b_\varepsilon^t(X_\varepsilon^t(x)).
\end{align*}
Hence, $ \det (\nabla_x X_\varepsilon^t(x)) = \exp \left( \int_0^t \mathrm{div} \, b_\varepsilon^s(X_\varepsilon^s(x)) ds\right)$. Furthermore, the divergence of $b_\varepsilon^t$ is bounded from below by a constant. Indeed, 
\begin{equation*}
    \nabla_x b_\varepsilon^t = - \nabla^2 H(\nabla_x f_\varepsilon^t) \nabla_x^2 f_\varepsilon^t,
\end{equation*}
and according to the hypothesis of Theorem \ref{t.wuniqueness}, there exists a positive constant $c_0$ such that $x \mapsto c_0|x|^2-f(t,x)$ is convex for all $t \geq 0$. So, for all $\varepsilon > 0 $, $\nabla_x^2 f_\varepsilon \leq c_0$. The function $H$ is convex, so $\nabla^2 H(\nabla_x f_\varepsilon^t(x))$ is a positive semi-definite matrix for all $x$. Let $M(x)$ denote the square root of this matrix, we have 

\begin{align*}
    \mathrm{div} \,  b_\varepsilon^t(x) &= \mathrm{Tr} (\nabla_x b_\varepsilon^t(x)) \\
    &= \mathrm{Tr}( - \nabla^2 H(\nabla_x f_\varepsilon^t(x)) \nabla_x^2 f_\varepsilon^t(x)) \\
    &= \mathrm{Tr}(M(x)(c_0-\nabla^2_xf_\varepsilon^t(x))M(x)^*) - c_0 \mathrm{Tr}(\nabla^2H(\nabla_x f_\varepsilon^t(x))) \\
    &\geq 0 - c_0 \Delta H(\nabla_x f^t_\varepsilon(x)) \\
    &\geq -c,
\end{align*}
where $c = c_0 \|\Delta H\|_{L^\infty(B(0,\mathrm{Lip}(f)))}$. In conclusion,
\begin{equation*}
    \det (\nabla_x X_\varepsilon^t(x))^{-1} = \exp \left( -\int_0^t \mathrm{div} \, b_\varepsilon^s(X_\varepsilon^s(x)) ds\right) \leq e^{ct}.
\end{equation*}
Thus, the desired bound is proven. The second inequality, $ \mathcal{L}^d((X_\varepsilon^t)^{-1}(A)) \leq e^{ct} \mathcal{L}^d(A)$ is obtained by taking continuous positive approximations of $1_A$ in \eqref{e.anticoncentration} and using dominated convergence. In particular, if $\mathcal{L}^d(A) = 0$ then $\mathcal{L}^d((X_\varepsilon^t)^{-1}(A)) = 0$, so $(X_\varepsilon)_*\mathcal{L}^d \ll \mathcal{L}^d$.
\end{proof}

We will now tackle the problem of showing that the sequence $(X_\varepsilon)$ converges up to extraction. Let $E$ be a topological space, we denote by $\mathcal{P}(E)$ the set of Borel probability measures on $E$. We recall the following basic result on the theory of Young measures \cite[Theorem 6.5]{bvequation}.

\begin{theorem} \label{t.youngmeasure}
    Let $K$ be a compact metric space and $A \subset \mathbb{R}^d$ be a measurable set. Let $(\eta_h)_h$ be a sequence of measurable maps $A \to \mathcal{P}(K)$. There exists an extraction $(\varepsilon_k)_k$ and a measurable map $\eta : A \to \mathcal{P}(K)$ such that the following holds. For every bounded function $\phi = \phi(x,u) : A \times K \to \mathbb{R}$ that is continuous with respect to $u$ and measurable with respect to $x$, we have 
    \begin{equation}
        \lim_{k \to \infty} \int_A \int_K \phi(x,u) d\eta_{\varepsilon_k,x}(u) dx = \int_A \int_K \phi(x,u) d\eta_{x}(u)dx.
    \end{equation}
\end{theorem}

Fix $T \in (0,T^*)$, and define $\Gamma = \mathcal{C}^0([0,T],\mathbb{R}^d)$. By the Arzelà-Ascoli theorem, the closure $\mathcal{K}$ of $\{ X_\varepsilon(\cdot,x), x \in B(0,R), \varepsilon > 0\}$ in $\Gamma$ is a compact set for the topology of uniform convergence. For every $\varepsilon >0$, consider
\begin{equation}
\eta_\varepsilon : \begin{cases}
  \mathbb{R}^d \longrightarrow \mathcal{P}(\mathcal{K}) \\
   x \longmapsto \delta_{X_\varepsilon(\cdot,x)}.
\end{cases}
\end{equation}
By Theorem \ref{t.youngmeasure}, up to extraction, the sequence $(\eta_\varepsilon)_\varepsilon$ converges to some measurable map $\eta : \mathbb{R}^d \to \mathcal{P}(\mathcal{K})$. More precisely, there exists a sequence $(\varepsilon_k)_k$ such that for every bounded function $\phi = \phi(x,\gamma) : \mathbb{R}^d \times \Gamma \to \mathbb{R}$ that is continuous in $\gamma$ and measurable in $x$, we have 
\begin{equation}
  \lim_{k \to \infty} \int_{\mathbb{R}^d} \phi(x,X_{\varepsilon_k}(\cdot,x))dx = \int_{\mathbb{R}^d} \int_\Gamma \phi(x,\gamma) d\eta_x(\gamma)dx.
\end{equation}
The remainder of this section is devoted to showing that for almost all $x \in \mathbb{R}^d$, the probability measure $\eta_x$ is well-behaved and concentrated on the set of integral solutions of \eqref{e.chareqz} with initial condition $x$.
\begin{lemma} \label{l.limitanticoncentration}
  Let $c > 0$ be the constant of Lemma \ref{l.anticoncentration}. For every $t \in [0,T]$, $R \in (0,\infty]$ and non-negative continuous bounded map $\beta : \mathbb{R}^d \to \mathbb{R}$, we have
  \begin{equation} \label{e.limitbound}
    \int_{B_R} \int_\Gamma \beta(\gamma(t)) d\eta_x(\gamma) dx \leq e^{ct} \int_{B_{M_R}} \beta(x)dx,
  \end{equation}
  where $M_R = R + T|b|_\infty$. Furthermore, if $A \subset \mathbb{R}^d$ is a measurable set, then 
  \begin{equation} \label{e.limitneg}
     \int_{\mathbb{R}^d} \int_\Gamma 1_A(\gamma(t)) d\eta_x(\gamma)dx \leq e^{ct} \mathcal{L}^d(A).
  \end{equation}
\end{lemma}

\begin{proof}
  Consider, $\phi(x,\gamma) = \beta(\gamma(t))$, then using Lemma \ref{l.anticoncentration} we have,
  \begin{equation}
    \int_{B_R} \phi(x,X_\varepsilon(\cdot,x))dx = \int_{B_R} \beta(X_\varepsilon(t,x))dx \leq e^{ct} \int_{B_{R+t|b|_\infty}} \beta(x) dx.
  \end{equation}
  Letting $\varepsilon \to 0$ along the subsequence $(\varepsilon_k)_k$ we obtain \eqref{e.limitbound}. Let $A \subset \mathbb{R}^d$ be a measurable set, still according to Lemma \ref{l.anticoncentration}, we have 
  \begin{equation}
    \int_{\mathbb{R}^d} \int_\Gamma 1_A(\gamma(t)) d\eta_{\varepsilon,x} dx = \int_{\mathbb{R}^d} 1_A(X_\varepsilon^t(x)) dx \leq e^{ct}\mathcal{L}^d(A).
  \end{equation}
  Again, letting $\varepsilon \to 0$ along the subsequence $(\varepsilon_k)_k$ we obtain \eqref{e.limitneg}.
\end{proof}

For $x \in \mathbb{R}^d$, let $\Gamma_x^b$ denote the set of integral solutions of \eqref{e.chareqz} with initial condition $x$, that is
\begin{equation}
  \Gamma_x^b = \{ \gamma \in \Gamma, \forall t \in [0,T], \gamma(t)= x + \int_0^t b(\tau,\gamma(\tau))d\tau \}.
\end{equation}
\begin{lemma}
  For $\mathcal{L}^d$-almost every $x \in \mathbb{R}^d$, we have $\eta_x(\Gamma - \Gamma_x^b) = 0$.
\end{lemma}

\begin{proof}
  Fix $t \in [0,T]$ and let $c : \mathbb{R}_+ \times \mathbb{R}^d \to \mathbb{R}^d$ be a bounded smooth vector field. Define
  \begin{equation}
    \phi^c(x,\gamma) = \left| \gamma(t) - x - \int_0^t c(\tau, \gamma(\tau))d\tau \right|.
  \end{equation}
  Let $\varepsilon > 0$ and $R > 0$, using Fubini's theorem and Lemma \ref{l.anticoncentration}, we have 
  \begin{align*}
    \int_{B_R} \phi^c(x,X_\varepsilon(\cdot,x))dx &\leq \int_{B_R} |X_\varepsilon(t,x)-x - \int_0^t c(\tau,X_\varepsilon(\tau,x))d\tau |dx \\
    &\leq \int_{B_R} \int_0^t |b_\varepsilon-c|(\tau,X_\varepsilon(\tau,x))d\tau dx \\
    &\leq \int_0^t \int_{B_R} |b_\varepsilon-c|(\tau,X_\varepsilon(\tau,x))dx d\tau \\
    &\leq \int_0^t e^{c\tau} \int_{B_{M_R}} |b_\varepsilon-c|(\tau,x)dx d\tau \\
      &\leq e^{ct} \int_0^t |b^\tau_\varepsilon-c^\tau|_{L^1(B_{M_R})} d\tau.
    \end{align*}
    By dominated convergence, letting $\varepsilon \to 0$ along the subsequence $(\varepsilon_k)_k$ and using the definition of the map $\eta$, we obtain 
    \begin{equation}
        \int_{B_R} \int_\Gamma \phi^c(x,\gamma) d\eta_x(\gamma) dx \leq e^{ct} \int_0^t |b^\tau-c^\tau|_{L^1(B_{M_R})} d\tau.
    \end{equation}
    Taking $c = b_\varepsilon$ leads to 
    \begin{equation}
        \lim_{\varepsilon \to 0} \int_{B_R} \int_\Gamma \phi^{b_\varepsilon}(x,\gamma) d\eta_x(\gamma)dx = 0.
    \end{equation}
    By Fatou's lemma, 
    \begin{equation}
        \int_{B_R} \liminf_{\varepsilon \to 0} \int_\Gamma \phi^{b_\varepsilon}(x,\gamma) d\eta_x(\gamma) dx = 0.
    \end{equation}
    So, for $\mathcal{L}^d$-almost every $x \in B_R$,
    \begin{equation} \label{e.liminfk}
        \liminf_{\varepsilon \to 0} \int_\Gamma \phi^{b_\varepsilon} (x,\gamma) d\eta_x(\gamma) = 0.
    \end{equation}
    Letting $R \to \infty$ along a countable set, we obtain that \eqref{e.liminfk} is true for $\mathcal{L}^d$-almost every $x \in \mathbb{R}^d$.
    
    Let $N \subset (0,T) \times \mathbb{R}^d$ be the complement of $\overline{N}  = \{ (t,x), \lim_{k \to \infty} b_{\varepsilon_k}(t,x) = b(t,x) \}$. For every $t \in (0,T)$, define $N_t = \{x \in \mathbb{R}^d, (t,x) \in N \}$. The set $N$ is $\mathcal{L}^{d+1}$-negligible, so for $\mathcal{L}^1$-almost every $t \in (0,T)$, $\mathcal{L}^d(N_t) = 0$. For every $t \in (0,T)$ such that, $\mathcal{L}^d(N_t) = 0$, according to Lemma \ref{l.limitanticoncentration} we have $ \int_{\mathbb{R}^d} \int_\Gamma 1_{N_t}(\gamma(t)) d\eta_x(\gamma)dx = 0$. Therefore, 
  \begin{equation}
    \int_0^T \int_{\mathbb{R}^d} \int_\Gamma 1_{N_t}(\gamma(t))d\eta_x(\gamma) dx dt= 0.
  \end{equation}
  By Fubini's theorem, 
  \begin{equation}
    \int_{\mathbb{R}^d} \int_\Gamma \int_0^T 1_{N_t}(\gamma(t)) dt d\eta_x(\gamma) dx = 0.
  \end{equation}
  So for $\mathcal{L}^d$-almost every $x \in \mathbb{R}^d$,
  \begin{equation} \label{e.etax}
    \eta_x \left( \left\{ \gamma \in \Gamma, \int_0^T 1_{N_t}(\gamma(t))dt = 0 \right\}\right) = 1.
  \end{equation}
   Now, let $x \in \mathbb{R}^d$ be such that \eqref{e.liminfk} and \eqref{e.etax} are satisfied, and let $\gamma \in \Gamma$ be a point in the support of $\eta_x$ such that $\int_0^T 1_{N_t}(\gamma(t))dt = 0$. Then, for $\mathcal{L}^1$-almost every $t \in (0,T)$, we have $\gamma(t) \notin N_t$, thus for $\mathcal{L}^1$-almost every $t \in (0,T)$,
    \begin{equation}
        \lim_{k \to \infty} b_{\varepsilon_k}(t,\gamma(t)) = b(t,\gamma(t)). 
    \end{equation} 
    By dominated convergence, it follows that $ \lim_{k \to \infty} \int_0^T b_{\varepsilon_k}(t,\gamma(t)) dt = \int_0^T b(t,\gamma(t))dt$. And thus,
    \begin{equation} \label{e.limittime}
       \lim_{k \to \infty} |\gamma(t) - x - \int_0^t b_{\varepsilon_k}(\tau,\gamma(\tau)) d\tau |= |\gamma(t) - x - \int_0^t b(\tau,\gamma(\tau)) d\tau |.
    \end{equation}
    Let $(\varepsilon_{k(l)})_l$ be a subsequence of $(\varepsilon_k)$ along which the $\liminf$ in \eqref{e.liminfk} is reached. Combining \eqref{e.liminfk} with \eqref{e.limittime}, and taking the limit as $l \to \infty$, we obtain
    \begin{equation}
        \int_\Gamma |\gamma(t) - x - \int_0^t b(\tau, \gamma(\tau))d\tau|d\eta_x(\gamma)  = 0.
    \end{equation}
    In conclusion, for $\mathcal{L}^d$-almost every $x \in \mathbb{R}^d$,
    \begin{equation}
        \eta_x \left( \left\{\gamma \in \Gamma, \gamma(t) = x + \int_0^t b(\tau,\gamma(\tau))d\tau \right\} \right) = 1.
    \end{equation}
    Thus, for $\mathcal{L}^d$-almost every $x \in \mathbb{R}^d$,
    \begin{equation}
        \eta_x \left( \left\{\gamma \in \Gamma, \forall t \in [0,T] \cap \mathbb{Q}, \gamma(t) = x + \int_0^t b(\tau,\gamma(\tau))d\tau \right\} \right) = 1.
    \end{equation}
    Since, the functions in $\Gamma$ are continuous with respect to $t$ and $[0,T] \cap \mathbb{Q}$ is dense in $[0,T]$, we have,
    \begin{equation}
        \Gamma_x^b = \left\{\gamma \in \Gamma, \forall t \in [0,T] \cap \mathbb{Q}, \gamma(t) = x + \int_0^t b(\tau,\gamma(\tau))d\tau \right\}.
    \end{equation}
    Finally, we obtain $\eta_x(\Gamma - \Gamma_x^b) = 0$ for $\mathcal{L}^d$-almost every $x \in \mathbb{R}^d$.
\end{proof}

\section{Proof of the main result} \label{s.endproof}

The goal of this section is to use the Young measure $\eta$ built in Section \ref{s.constructionofx} to complete the proof of Theorem \ref{t.wuniqueness}. Recall that in Section \ref{s.firstineq} we have already shown the existence of $T \in (0,T^*)$ such that $f \geq u$ on $[0,T) \times \mathbb{R}^d$. Also recall that $u$ is $\mathcal{C}^{1,1}$ on $[0,T) \times \mathbb{R}^d$ and $W^t$ is a surjective function for $t < T$. Here we will try to prove the converse inequality $f \leq u$ using an argument similar to the one used to prove Lemma \ref{l.firstineq} in which the vector field $W$ will be replaced by $\eta$. This approach will only lead to a bound valid on the image of $(t,x) \mapsto (t,\gamma(t))$ for $\mathcal{L}^d$-almost every $x \in \mathbb{R}^d$ and $\eta_x$-almost every $\gamma \in \Gamma_x^b$. But, using this bound, we will be able to show that $\eta_x = \delta_{W(\cdot,x)}$, except on a negligible set. Then, using the surjectivity of $W^t$ for $t < T$, we can finish the proof of Theorem \ref{t.wuniqueness}. 

\begin{lemma} \label{l.f-u=0}
    For $\mathcal{L}^d$-almost every $x \in \mathbb{R}^d$, and $\eta_x$-almost every $\gamma \in \Gamma_x^b$, we have
    \begin{equation}
        \forall t \in [0,T], \; f(t,\gamma(t))= u(t,\gamma(t)).
    \end{equation}
\end{lemma}

\begin{proof}
    Fix $t \in (0,T)$, let 
    \begin{align*}
        \mathcal{D}_t &= \left\{ x \in \mathbb{R}^d, \textrm{$f$ is not differentiable at $(t,x)$}\right\} \subset \mathbb{R}^d \\
        A_t &= \left\{ (x,\gamma) \in \mathbb{R}^d \times \Gamma, \textrm{$f$ is not differentiable at $(t,\gamma(t))$} \right\} \subset \mathbb{R}^d \times \Gamma.
    \end{align*}
    By Rademacher's theorem, for $\mathcal{L}^1$-almost every $t \in (0,T)$, the set $\mathcal{D}_t$ is $\mathcal{L}^d$-negligible. Hence, according to Lemma \ref{l.limitanticoncentration}, for $\mathcal{L}^1$-almost every $t \in (0,T)$, the following holds, 
    \begin{equation} \label{e.f-uneg}
        \int_{\mathbb{R}^d} \int_\Gamma 1_{A_t}(x,\gamma) d\eta_x(\gamma) dx = \int_{\mathbb{R}^d} \int_\Gamma 1_{\mathcal{D}_t} (\gamma(t)) d\eta_x(\gamma)dx = 0.
    \end{equation}
    Integrating \eqref{e.f-uneg} with respect to $t \in (0,T)$ and invoking Fubini's theorem, we obtain 
    \begin{equation}
         \int_{\mathbb{R}^d} \int_\Gamma \int_0^T 1_{A_t}(x,\gamma) d\eta_x(\gamma) dx = \int_0^t \int_{\mathbb{R}^d} \int_\Gamma 1_{\mathcal{D}_t} (\gamma(t)) dt d\eta_x(\gamma)dx = 0.
    \end{equation}
    So, for $\mathcal{L}^d$-almost every $x \in \mathbb{R}^d$ and $\eta_x$-almost every $\gamma \in \Gamma_x^b$, $\int_0^T 1_{A_t}(x,\gamma) dt = 0$. In particular, for $\mathcal{L}^d$-almost every $x \in \mathbb{R}^d$ and $\eta_x$-almost every $\gamma \in \Gamma_x^b$, the function $t \mapsto (f-u)(t,\gamma(t))$ is differentiable $\mathcal{L}^1$-almost everywhere. Fix $x \in \mathbb{R}^d$ and $\gamma \in \Gamma_x^b$ such that for $\mathcal{L}^1$-almost every $t \in (0,T)$, $f$ is differentiable at $(t,\gamma(t))$. We have, for $\mathcal{L}^1$-almost every $t \in (0,T)$,
    \begin{align*}
        \frac{d}{dt} (f-u)(t,\gamma(t)) &= \partial_t f(t,\gamma(t)) - \partial_t u(t,\gamma(t)) + \gamma'(t) \cdot (\nabla f(t,\gamma(t)) - \nabla u(t,\gamma(t))) \\
        &= H(\nabla f) - H(\nabla u) - \nabla H(\nabla f) \cdot (H(\nabla f) - H(\nabla u ))(t,\gamma(t)) \\
        &= -(T_{\nabla f} H(\nabla u ) - H(\nabla u))(t,\gamma(t)) \\
        &\leq 0,
    \end{align*}
    where $T_pH(q) = H(p) + \nabla H(p) \cdot(q-p)$ is the tangent of $H$ at $p$ evaluated at $q$. For this choice of $x$ and $\gamma$, the function $t \mapsto (f-u)(t,\gamma(t))$ vanishes at zero and is non-increasing. In addition, according to Lemma \ref{l.firstineq} the function $t \mapsto (f-u)(t,\gamma(t))$ is non-negative. Therefore, for every $t \in (0,T)$, $f(t,\gamma(t)) = u(t,\gamma(t))$.
\end{proof}

\begin{lemma} \label{l.nablaeq}
     For $\mathcal{L}^d$-almost every $x \in \mathbb{R}^d$, $\eta_x$-almost every $\gamma \in \Gamma_x^b$, and $\mathcal{L}^1$-almost every $t \in (0,T)$, the function $f$ is differentiable at $(t,\gamma(t))$ and we have
    \begin{equation}
        \nabla f(t,\gamma(t)) = \nabla u(t,\gamma(t)).
    \end{equation}
\end{lemma}

\begin{proof}
    Let $x \in \mathbb{R}^d$ and $\gamma \in \Gamma_x^b$ such that the previous lemma holds. The function $t \mapsto (f-u)(t,\gamma(t))$ is constant equal to zero, therefore it is differentiable, and its derivative is equal to zero. In addition, inspecting the proof of the previous lemma, we see that for $\mathcal{L}^1$-almost all $t \in (0,T)$, $f$ and $u$ are differentiable at $(t,\gamma(t))$. For every $t \in (0,T)$ such that $f$ and $u$ are differentiable at $(t,\gamma(t))$, we have
    \begin{equation}
        0 = \frac{d}{dt} (f-u)(t,\gamma(t)) =  -(T_{\nabla f} H(\nabla u ) - H(\nabla u))(t,\gamma(t)).
    \end{equation}
    By strict convexity of $H$, we obtain $\nabla f (t,\gamma(t)) = \nabla u(t,\gamma(t))$.
\end{proof}

\begin{proof}[Proof of Theorem \ref{t.wuniqueness}]
Let $x \in \mathbb{R}^d$ and $\gamma \in \Gamma_x^b$ be such that Lemma \ref{l.nablaeq} holds. We have for every $t \in (0,T)$,
\begin{align*}
    \gamma(t) &= x - \int_0^t \nabla H(\nabla f (\tau, \gamma(\tau)))d\tau \\
    &= x - \int_0^t \nabla H(\nabla u (\tau, \gamma(\tau)))d\tau.
\end{align*}
Therefore $\gamma$ is an integral solution of \eqref{e.ODEu} with initial condition $x$, so by Lemma \ref{l.unicity} we have $\gamma = W(\cdot,x)$. So, for $\mathcal{L}^d$-almost every $x \in \mathbb{R}^d$, $\eta_x$ is supported on $\{W(\cdot,x)\}$, this means that $\eta_x = \delta_{W(\cdot,x)}$. Applying Lemma \ref{l.f-u=0} we obtain, for $\mathcal{L}^d$-almost every $x \in \mathbb{R}^d$ and every $t \in (0,T)$ that,
\begin{equation}
    f(t,W(t,x)) = u(t,W(t,x)).
\end{equation}
By continuity of $W$, $f$ and $u$, the identity in the above display is actually true for every $x \in \mathbb{R}^d$. Finally, according to Lemma \ref{l.propw}, for every $t \in [0,T)$ the map $W(t,\cdot)$ is surjective, thus $f=u$ on $[0,T) \times \mathbb{R}^d$.  
\end{proof}

\section{A counter-example for non-convex nonlinearities} \label{s.counterexample}

The aim of this section is to show that when the nonlinearity is not assumed to be convex, Theorem \ref{t.uniqueness} no longer holds. We will consider the conservation law naturally associated with \eqref{e.HJ}. In order to keep our notation consistent with the literature on conservation laws, we choose to write this equation with the $+$ sign convention rather than the possibly more natural choice of writing it with the $-$ sign convention. To construct a counter-example in the absence of convexity, we will construct a Lipschitz initial condition $v_0 : \mathbb{R} \to \mathbb{R}$ and a $\mathcal{C}^2$ non-convex $H : \mathbb{R} \to \mathbb{R}$ such that the initial value problem
\begin{equation} \label{e.conservation}
    \begin{cases} 
\partial_t v + \partial_x \left( H(v) \right) = 0, \\ v(0,\cdot) = v_0,
\end{cases}
\end{equation} 
admits a non-entropy solution. We say that a function $\phi : \mathbb{R} \to \mathbb{R}$ is one-sided Lipschitz when there exists $c \geq 0$ such that, for all $x \geq y$ :
\begin{equation}
    \phi(x) - \phi(y) \geq - c(x-y). 
\end{equation}
There is a correspondence between the solutions of \eqref{e.conservation} and the solutions of the following Hamilton-Jacobi initial value problem,
\begin{equation} \label{e.hj}
    \begin{cases} 
\partial_t u - \overline{H}( \partial_x u)  = 0, \\ u(0,\cdot) = u_0,
\end{cases}
\end{equation} 
where $\overline{H}(p) = H(-p)$. Here, the nonlinearity we build is even, so we simply have $\overline{H} = H$. The correspondence is given by $v \mapsto -\int_{-\infty}^\cdot v(\cdot,y)dy$ and $u \mapsto -\partial_x u$. This correspondence sends the viscosity solution of \eqref{e.hj} to the entropy solution of \eqref{e.conservation} and vice versa \cite[Theorem 16.1]{PLL}, \cite[Section 2]{viscocl}. Hence, building a bounded, one-sided Lipschitz non-entropy solution $g$ of \eqref{e.conservation} with Lipschitz initial condition yields a semi-concave non-viscosity solution $f$ of \eqref{e.hj} with $\mathcal{C}^{1,1}$ initial condition. This function $f$ satisfies all the hypotheses of Theorem \ref{t.uniqueness} but is not a viscosity solution.

We define a non-convex even function $H$ by
\begin{equation} \label{e.hamiltonian}
   H(p) =  \begin{cases}
     \frac{5}{4}p^3+ \frac{19}{8} p^2+ \frac{15}{16} p+\frac{5}{32} &\textrm{ if } p \leq -1/2 \\ 
    \frac{1}{2}p^2 &\textrm{ if } -1/2 \leq p \leq 1/2 \\ 
     -\frac{5}{4}p^3+ \frac{19}{8} p^2 - \frac{15}{16} p+\frac{5}{32} &\textrm{ if } 1/2 \leq p, 
    \end{cases}
\end{equation}
and a Lipschitz initial condition
\begin{equation}
    v_0(x) = \begin{cases}
    -3/2 &\textrm{ if } x \leq -3/2 \\
    x &\textrm{ if } -3/2 \leq x \leq 3/2 \\
    3/2 &\textrm{ if } 3/2 \leq x \leq L \\
    -x + L +3/2 &\textrm{ if } L \leq x \leq L+1\\
    1/2 &\textrm{ if } L+1\leq x.
    \end{cases}
\end{equation}
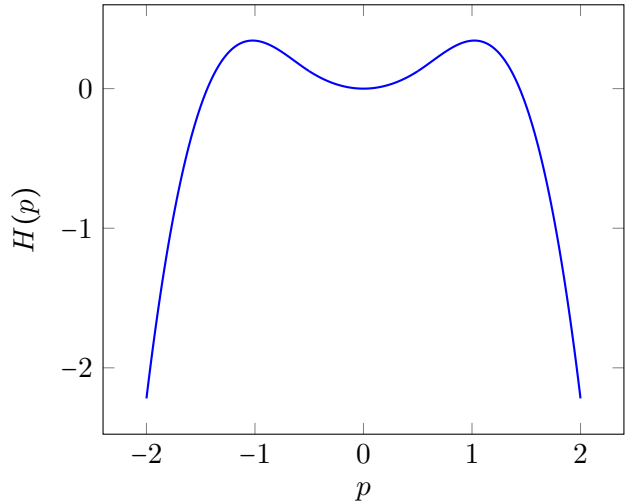
\begin{figure}
    \begin{tikzpicture}
    \begin{axis}[xlabel = {$p$},ylabel = {$H(p)$}]
         \addplot[domain = -2:-1/2, 
         samples = 200,
        smooth,
        thick,
        blue,] {5/4*x^3 + 19/8*x^2 + 15/16*x+5/32};
        \addplot[domain = -1/2:1/2, 
         samples = 200,
        smooth,
        thick,
        blue,] {1/2*x^2};
        \addplot[domain = 1/2:2, 
         samples = 200,
        smooth,
        thick,
        blue,] {-5/4*x^3 + 19/8*x^2 - 15/16*x+5/32};
    \end{axis}
\end{tikzpicture}
    \caption{Graph of the non-convex function $H$}
    \label{f.H}
\end{figure}
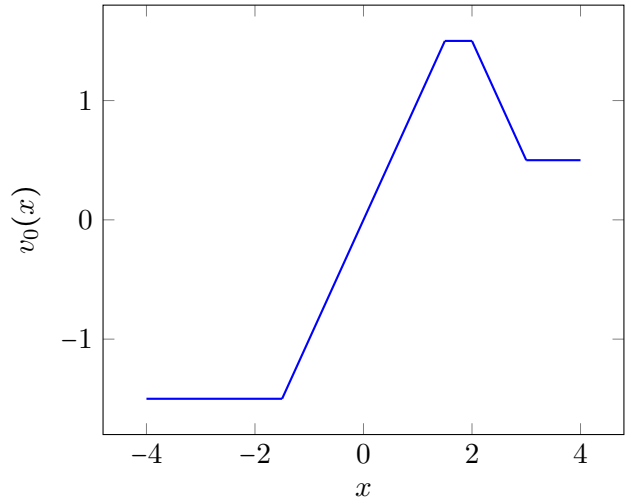
\begin{figure}
    \begin{tikzpicture}
    \begin{axis}[xlabel = {$x$},ylabel = {$v_0(x)$}]
         \addplot[domain = -4:-3/2, 
         samples = 200,
        smooth,
        thick,
        blue,] {-3/2};
        \addplot[domain = -3/2:3/2, 
         samples = 200,
        smooth,
        thick,
        blue,] {x};
        \addplot[domain = 3/2:2, 
         samples = 200,
        smooth,
        thick,
        blue,] {3/2};
        \addplot[domain = 2:3, 
         samples = 200,
        smooth,
        thick,
        blue,] {-x+7/2};
        \addplot[domain = 3:4, 
         samples = 200,
        smooth,
        thick,
        blue,] {1/2};
    \end{axis}
\end{tikzpicture}
    \caption{Graph of the Lipschitz initial condition $v_0$ for $L = 2$}
    \label{f.v_0}
\end{figure}
\noindent The exact value of the parameter $L$ is to be fixed later, for now we imagine that $L$ is large, for example $L > 30 \times H'(-3/2) $, so that the effect of the asymmetry of $v_0$ is felt close to the origin only after a long enough time. The characteristic curves $(t,X^t(x))$ of \eqref{e.conservation} satisfy $X^t(x) = x + tH'(v_0(x))$. Setting $a = H'(-3/2) = -H'(3/2) > 0$, explicitly we have,
\begin{equation}
    X^t(x) = \begin{cases}
    x + ta &\textrm{ if } x \leq -3/2\\
    x+tH'(x) &\textrm{ if } -3/2 \leq x \leq -1/2\\
    x+tx &\textrm{ if } -1/2 \leq x \leq 1/2 \\
   x+tH'(x) &\textrm{ if } 1/2 \leq x \leq 3/2 \\
    x-ta &\textrm{ if } 3/2 \leq x \leq L \\
    x+tH'(3/2 + L-x) &\textrm{ if } L \leq x \leq L+1 \\
    x+tH'(1/2) &\textrm{ if } L+1 \leq x
    \end{cases}
\end{equation}
\begin{figure}
    \centering
    \includegraphics[scale = 0.25]{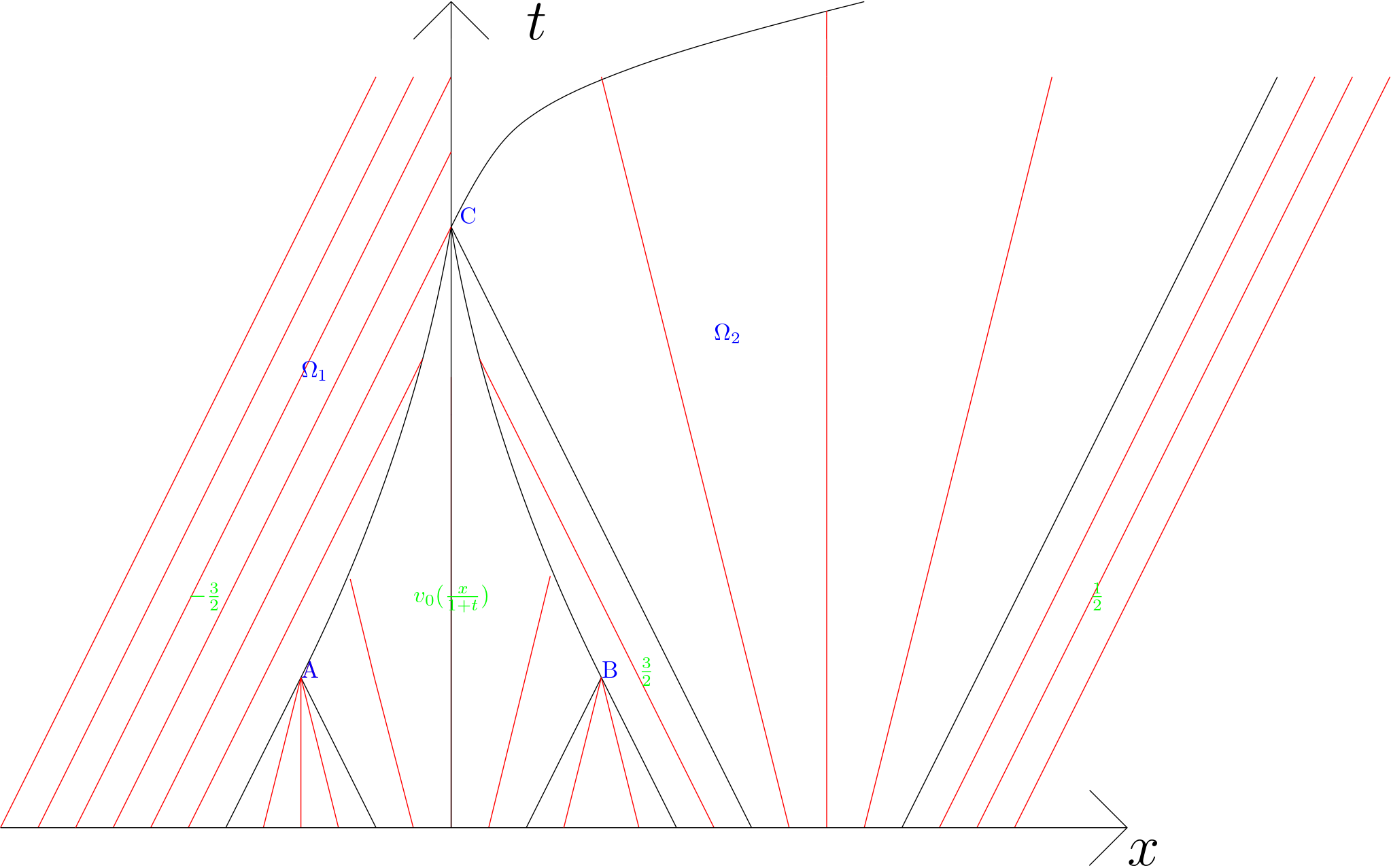}
    \caption{Graph of $(x,t) \mapsto (x+tH(v_0(x)),t)$ when $L=at_1$}
    \label{fig:my_label}
\end{figure}
\noindent The characteristic lines are drawn on Figure \ref{fig:my_label}, along with three curves started at the points $A$, $B$ and $C$. we can build a solution $v$ of \eqref{e.conservation} using the characteristic lines. Given $(t,y) \in \mathbb{R}_+\times \mathbb{R}$ we set $v(t,y) = v_0(x)$ where $y = X^t(x)$ and $s \mapsto X^s(x)$ is the unique characteristic curve that reaches $(t,y)$ on Figure \ref{fig:my_label}.

We start by checking that the solution $v$ we have built satisfies the Rankine-Hugoniot condition and that the curves along which the discontinuities evolve look like the curves on figure \ref{fig:my_label}. The function $v$ exhibits discontinuities, they appear at time $t = t_0$ at $x = -x_0$ and $x = x_0$ where $x_0 = 3/2-  \frac{a}{a+1/2} \simeq  0.68$ and $t_0 = \frac{1}{a+1/2} \simeq 0.36$ (point $A$ and $B$ on Figure \ref{fig:my_label}). Those discontinuities continue to exist at times $t > t_0$ along some curves $\Gamma_A = \{ (t,z_A(t)) \}$ and $\Gamma_B = \{ (t,z_B(t)) \}$. Furthermore, $z_A$ and $z_B$ satisfy the Rankine-Hugoniot condition, that is 
\begin{equation}
    [H(v)]_A = \frac{dz_A}{dt} [v]_A,
\end{equation}
\begin{equation}
    [H(v)]_B = \frac{dz_B}{dt} [v]_B,
\end{equation}
where $[\cdot]_A$ and $[\cdot]_B$ respectively denote the jump of $\cdot$ along $\Gamma_A$ and $\Gamma_B$. Recall that if $g$ is a piecewise continuous function whose discontinuities lie along a curve $\Gamma = \{ (z(t),t)\}$, then the jump of $g$ along $\Gamma$ denoted $[g]$ is defined by 
\begin{equation}
    [g] = \lim_{\omega_+} g  -  \lim_{\omega_-} g,
\end{equation}
where $\omega_- = \{ x < z(t)\}$ and $\omega_+ = \{ x > z(t)\}$ for simplicity we write $g_\pm$ instead of $\lim_{\omega_\pm} g$. Using the expression of $v$ in terms of $v_0$ in each region,
\begin{align}
    \frac{dz_A}{dt} = \frac{ H((\frac{z_A}{1+t})) - H(-3/2)}{\frac{z_A}{1+t}- (-3/2)} \geq 0, \\ 
    \frac{dz_B}{dt} = \frac{H(3/2) - H(\frac{z_B}{1+t}) }{3/2 - \frac{z_B}{1+t}} \leq 0.
\end{align}
In particular, $z_A$ is a non-decreasing function of time, meaning $\Gamma_A$ is curved to the right. Furthermore, let $g(p)$ denote the slope between $(-3/2,H(-3/2))$ and $(p,H(p))$, we have $z'_A(t) = g(z_A(t))$, in addition there exists $m > 0$ such that for all $p \in [-x_0,0]$ we have $g(p) \geq m$. In particular, $z_A'(t) \geq m > 0$ so $z_A$ goes to $0$ at a speed greater than some $m >0$. The curve $\Gamma_B$ is the symmetrical of the curve $\Gamma_A$ with respect to the vertical axis, so similar observations can be made for the function, $z_B$ in particular $z_B$ also goes to $0$ at a speed greater than $m > 0$. In particular, this shows that as time elapses, the two discontinuities move toward each other in the $(t,x)$-space and the speed at which they move toward each other is bounded from below by a positive constant. Thus, the two discontinuities meet at some point $C$ at a time, $t_1$ and by symmetry the point $C$ must lie on the vertical axis. Choosing $L = at_1$, the first characteristic started from $x \in [L,L+1]$ to hit the vertical axis is $X^t(L) = L + th(3/2)$ which hits the vertical axis at $t_2 = \frac{L}{-H'(3/2)} = t_1$. Hence, the influence of the asymmetry of the initial condition $v_0$ is felt by the discontinuities only after the moment they meet at the point $C$. Before time $t_1$ the discontinuities appearing at $A$ and $B$ behave as if the initial condition was
\begin{equation}
    v_0(x) = \begin{cases}
    -3/2 &\textrm{ if } x \leq -3/2, \\
    x &\textrm{ if } -3/2 \leq x \leq 3/2, \\
    3/2 &\textrm{ if } 3/2 \leq x. 
    \end{cases}
\end{equation}
In particular, by symmetry, the point $C$ must lie on the vertical axis. After time $t_1$ only one discontinuity remains and it evolves along the curve $\Gamma_C = \{ (z_C(t),t) \}$. The Rankine-Hugoniot condition imposes that after time $t_1$, the point $C$ evolves along the curve $\Gamma_C  = \{(z_C(t),t) \}$ where $z_C$ satisfies
\begin{equation}
    [H(v)]_C = \frac{dz_C}{dt} [v]_C.
\end{equation}
After time $t_1$, along $\Gamma_C$ we have $v_+ = 3/2 + L - x $ for some $x \in [L,L+1]$ so $v_+ \in [1/2,3/2]$ hence $z_C'(t)$ is non-negative and $\Gamma_C$ leans to the right. As the discontinuity deviates to the right, $v_+$ becomes closer to $1/2$ and $z_C'(t)$ becomes bigger. Thus, the speed at which the discontinuity moves away from the vertical axis is an increasing function of time. Hence, at some time $t_3$ we have $z_C(t_3) = z$ where $z$ is the maximum of $H$ on $[1/2,3/2]$.

Note that the jumps along $\Gamma_A$, $\Gamma_B$ and $\Gamma_C$ of $v$ are all non-negative, so the space derivative of $v$ is bounded below (but not above) at all times by a constant that does not depend on time. In other words, there exists a constant $c$ such that the corresponding solution $f$ of \eqref{e.hj} satisfies : $x \mapsto c|x|^2 - f(t,x)$ is convex for all $t \geq 0$. Furthermore, for every, $(y,t) \in \mathbb{R} \times \mathbb{R}_+$ there exists $x \in \mathbb{R}$ such that $v(t,y) = v_0(x)$ so $v$ is a bounded function and $\|v\|_{L^\infty} \leq \| v_0 \|_{L^\infty}$. Let us now show that $v$ is not an entropy solution. By contradiction, suppose that $v$ is an entropy solution of \eqref{e.conservation}, the curve $\Gamma_C$ must satisfy the following additional condition \cite[Proposition 2.3.7]{conservationlaws},
\begin{equation} \label{e.entropy}
    [F(v)] \leq \frac{dz_C}{dt} [E(v)]
\end{equation}
for every entropy-entropy flux pair $(E,F)$, that is pairs of function $(E,F)$ such that $E$ is convex and $F' = H'E'$. It is well known that it is necessary and sufficient to check \eqref{e.entropy} only for pairs $(E,F)$ where $E = |x-k|$, $k \in \mathbb{R}$ for, $v$ to be an entropy solution \cite[Proposition 2.3.7]{conservationlaws}. In particular, $v$ is an entropy solution if and only if \eqref{e.entropy} holds for all pairs $(E,F)$ with $E = (x-k)_+$ and $E = (x-k)_-$, $k \in \mathbb{R}$. Hence, if $v$ is an entropy solution, we have for all $k \in (v_-,v_+)$,
\begin{equation} \label{e.slope}
    \frac{H(v_+) - H(k)}{v_+-k} \leq \frac{H(v_+) - H(v_-)}{v_+ - v_-} \leq \frac{H(k) - H(v_-)}{k-v_-}.
\end{equation}
The condition \eqref{e.slope} is not satisfied by the solution $v$ that we built. Indeed, recall that $z$ denotes the maximum of $H$ on $[1/2,3/2]$, the point $(-3/2,H(-3/2))$ is strictly above the line defined by the points, $(0,H(0))$ and $(z,H(z))$ this means that, 
\begin{equation} \label{e.contradiction}
    \frac{H(z) -H(0)}{z-0} > \frac{H(z) -H(-3/2)}{z-(-3/2))}.
\end{equation}
But \eqref{e.contradiction} contradicts the inequality on the left-hand side of \eqref{e.slope} for $k = 0$ and time $t = t_3$. Hence, the solution $v$ that we have built is not an entropy solution and the associated solution of \eqref{e.hj} satisfies all the hypothesis of Theorem \ref{t.uniqueness} but is not a viscosity solution.

Finally, to highlight the challenge posed by this counter-example, let us define a two species spin model with a double-well covariance function resembling Figure \ref{f.H}. Let $\sigma = (\sigma_1,\sigma_2) \in \{-1,1\}^N \times \{-1,1\}^N$ and define 
\begin{equation}
    E_N(\sigma) = \frac{1}{\sqrt{N}}\sum_{i,j=1}^N J^1_{ij} \sigma_{1i} \sigma_{1j} + \frac{1}{\sqrt{N}} \sum_{i,j=1}^N J^2_{ij} \sigma_{2i} \sigma_{2j} + \frac{\sqrt{6}}{N^{3/2}}\sum_{i,j,k,l=1}^N J_{ijkl} \sigma_{1i} \sigma_{1j} \sigma_{2k} \sigma_{2l},
\end{equation}
where the $J^1_{ij}$'s, $J^2_{ij}$'s and the $J_{ijkl}$'s are iid standard Gaussian random variables. For every $\sigma, \tau \in \{-1,1\}^N \times \{-1,1\}^N$, we have 
\begin{equation} \label{e.model}
    \mathbb{E} \left[ E_N(\sigma) E_N(\tau) \right] = N \theta \left( \frac{\sigma_1 \cdot \tau_1}{N},  \frac{\sigma_2 \cdot \tau_2}{N}\right),
\end{equation}
where $\theta(x,y) = x^2+y^2+6x^2y^2$. The covariance function of the SK model is the square function, which is the function appearing in \eqref{e.HJW}. The same holds true for more general models, in the context of \eqref{e.model}, the behavior of the nonlinearity will be governed by the behavior of the function $\theta$. If we plot a slice of the function $\theta$ along the line $t \mapsto (1-t,t)$ we see that the function $\theta$ has a double-well structure. This is the same pathological behavior than the one exhibited by \eqref{e.hamiltonian}. The existence of a spin model with such a covariance function is problematic for generalizing this approach of Parisi formula in the non-convex case. It means, that the behavior of the counter-example we built cannot be easily ruled out for weak solutions arising in the context of spin glasses. Nonetheless, it was recently proven \cite[Theorem 1.1]{wavefront}, that even for non-convex models the limit free energy, if it exists, stays in the wavefront.

\begin{figure}[h]
    \begin{tikzpicture}[scale = 1.10]
    \begin{axis}[xlabel = {$t$},ylabel = {$-\theta(t,1-t)$},yticklabels={,,},xticklabels={,,}]
         \addplot[domain = 0:1, 
         samples = 200,
        smooth,
        thick,
        blue,]{-x^2-(1-x)^2-6*x^2*(1-x)^2};
    \end{axis}
\end{tikzpicture}
\end{figure}

\section*{Acknowledgement}
I warmly thank Jean-Christophe Mourrat, Pierre Cardaliaguet and Stefano Bianchini for the help they provided during the conception and the writing of this paper.

\bibliographystyle{abbrv}
{\small
\bibliography{ref}
}

\end{document}